\documentclass[12pt]{article}
\usepackage{graphics,amsmath,amssymb}
\usepackage{amsthm}
\usepackage{amsfonts}
\usepackage{latexsym}
\usepackage{amscd}

\newtheorem{theorem}{Theorem}[section]
\newtheorem{corollary}{Corollary}[section]

\begin{document}
\begin{center}
\title{The Expansion of each van der Waerden number $W(r, k)$ into Powers of $r$, when $r$ is the Number of Integer Colorings, determines a greatest lower Bound for all $k$ such that \(W(r, k) < r^{k^{2}}\)}
\author{\textbf{Robert J. Betts}}
\maketitle
\emph{The Open University\\Postgraduate Department of Mathematics and Statistics~\footnote{This paper was begun in 2010--2012 (but then revised subsequently from 2012--2015, after comments from a Reviewer), in part when the Author, who is a US citizen, was pursuing graduate studies in computer science at the University of Massachusetts Lowell, then postgraduate studies in mathematics at OU during 2012--2013. However none of the research for this paper was supported by funds or grants either from Open University or from UMASS, nor were these two academic institutions obligated in any way, to support this author's research.}\\ (Main Campus) Walton Hall, Milton Keynes, MK7 6AA, UK\\Robert\_Betts@alum.umb.edu}
\end{center}
\begin{abstract}
Every positive integer greater than a positive integer $r$ can be written as an integer that is the sum of powers of $r$. Here we use this to prove the conjecture posed by Ronald Graham, B. Rothschild and Joel Spencer back in the nineteen nineties, that the van der Waerden number with $r$ colorings and with arithmetic progressions of $k$ terms, has a certain upper bound. Our proof does not need the application of double induction, constructive methods of proof or combinatorics, as applied to sets of integers that contain some van der Waerden number as an element. The proof instead derives from certain \emph{a priori} knowledge that is known about any positive integer when the integer is large. The mathematical methods we use are easily accessible by those whose field of specialization lies outside of combinatorial number theory, such as discrete mathematics, computational complexity, elementary number theory or analytic number theory.
\end{abstract}
\section{Introduction and Apologia}
Let $N$ and $r$ be always two positive integers where $N$ is larger than $r$. Then the integer $N$ always has some expansion into integers 
$$
N = a_{n}r^{n} + a_{n - 1}r^{n - 1} + \cdots + a_{1}r + a_{0},
$$
that are expressed as powers of $r$, where \(a_{n} \in [1, r - 1]\), \(a_{n - 1}, \ldots, a_{0} \in [0, r - 1]\) and where the positive integer exponent $n$ is the least positive integer exponent for which \(r^{n} | N\) but for which $r^{n + 1}$ does not divide $N$ (See Theorem 2.1 and Proof, along with Eqtn. (10), and Table 2). This remains true even when $N$ is a van der Waerden number $W(r, k)$, where this is the smallest integer for which the integer interval $[1, W(r, k)]$ contains an arithmetic progression of $k$ arbitrary terms, where $r$ is the number of colors in the integer coloring in $[1, W(r, k)]$. To give a concrete example suppose \(r = 2\). Then for $W(2, 4)$ we get the inequality 
\begin{eqnarray}
2^{5}&\leq&W(2, 4)\nonumber\\
     &=   &35 = 1\cdot 2^{5} + 1\cdot 2^{1} + 1\cdot 2^{0} < 2^{6}.\nonumber
\end{eqnarray}
When one sums the finite power series expression for $W(2, 4)$ on the right hand side for the number $35$, one gets the value $35$ for the van der Waerden number $W(2, 4)$.\\
\indent Such an expansion of $W(r, k)$ into powers of $r$ can be used to establish a necessary condition for which the inequality \(W(r, k) < r^{k^{2}}\) is true, and by this we mean for all $k$ with a certain lower bound. All the presently known values for van der Waerden numbers, namely $W(2, 3)$, $W(2, 4)$, $W(2, 5)$, $W(2, 6)$, $W(3, 3)$, $W(3, 4)$ and $W(4, 3)$, confirm the results and techniques used in this paper (See Theorem 2.1, Theorem 2.2, Corollary 2.1, Corollary 2.2, Corollary 2.3, in Section 2 and Table 1, Table 2 and Table 3, Section 3).\\
\indent R. Graham, B. Rothschild and J. Spencer~\cite{Graham and Rothschild},~\cite{Graham and Spencer}, have conjectured that
$$
W(2, k) < 2^{k^{2}},
$$
where $W(2, k)$ is the van der Waerden number such that the interval $[1, W(2, k)]$ has some monochromatic $2$--coloring (e.g., ``red," ``blue") among certain integer elements, where those integer elements in the interval form among themselves an arithmetic progression of $k$ arbitrary terms. In this paper the conjecture is shown to be true, but we prove the more general case that 
$$
W(r, k) < r^{k^{2}},
$$
where \(W(r, k) > r\), which means that each van der Waerden number $W(r, k)$ lies always within some proper subset $[r^{n}, r^{n + 1}]$ of $[1, r^{k^{2}}]$, for all positive integers $k$ that have a certain lower bound. The truth of their conjecture follows when \(r = 2\). \\
\indent Here we present a very truthful fact we only can hope is well evident to any reader who is familiar with how to expand a positive integer $N$ into powers of a positive integer $r$, when \(N > r\), and whether one is familiar with double induction methods or not: For each $W(r, k)$, let $n$ be that positive integer exponent for which $r^{n}$ divides $W(r, k)$ but $r^{n + 1}$ does not divide $W(r, k)$. Then whenever each $W(r, k)$ is expanded into powers of $r$ the result proves that
$$
W(r, k) \in [r^{n}, r^{n + 1}], \: \forall \: W(r, k) \in [1, W(r, k)].
$$ 
\indent Our result uses the fact that \(W(r, k) > r\). To anyone not too familiar with van der Waerden numbers of this type, it just so happens that the strict inequality \(W(r, k) > r\) always is the case for each of these van der Waerden numbers, since each interval $[1, W(r, k)]$ cannot have a number $r$ of integer colorings that exceeds the value of $W(r, k)$. If it did that would mean that the integers in the interval $[1, W(r, k)]$ would have a total number $r$ of distinct colors that exceeds the actual cardinality of this integer set, which of course is impossible. Furthermore if we do not want to find trivial arithmetic progressions in the set, we also must rule out \(r = W(r, k)\), because if the total number $r$ of distinct colors possible for the coloring of the integers in the set $[1, W(r, k)]$ equals the actual cardinality of the integer set, then the set $[1, W(r, k)]$ cannot have any integer elements that have some monochromatic integer $r$--coloring for an arithmetic progression of arbitrary length \(k > 1\). For readers not familiar with van der Waerden numbers $W(r, k)$ to see why \(r > W(r, k)\) cannot hold, one might apply the Pigeonhole Principle to the set $[1, W(r, k)]$ where the $r$ different colors themselves are replaced by ``pigeons" and the integers in the set are different boxes. There would be ``boxes" that contain more than one ``pigeon." Thus if $r$ were greater than $W(r, k)$, there would be integers (boxes) that would get more than just one coloring (pigeon) within the interval. \\
\indent Why is the size of $r$ compared to the size of $W(r, k)$ so important? Because since \(W(r, k) > r\) is true always we can claim rightfully that, for each and every van der Waerden number $W(r, k)$, we always can use $r$ to give $W(r, k)$ some expansion into powers of $r$, which will help us to identify what is the real significance of the positive integer $n$ which actually is an exponent (See Section 2). It turns out this integer $n$ has major significance when it comes to determining the truth of the conjecture \(W(2, k) < 2^{k^{2}}\), and \emph{it is not necessary to know in advance, the value of each $n$ and each $W(r, k)$}, for the result to hold. A physicist might not rest until he or she finds his or her hypothesized Higgs boson or dark matter particle. For a mathematician however, usually it suffices, or should suffice, to show that some integer (like our integer exponent $n$) or integer property always exists. One does not need to find each and every individual prime of the form $4l + 1$ where $l$ always is any positive integer, for example, to know there are infinitely many such primes, a result we owe to P. L. Dirichlet~\cite{Apostol}. In the same vein and within this paper, one does not need to know beforehand the value of each and every $W(r, k)$, to know the integer exponent $n$ does exist for each van der Waerden number $W(r, k)$ as indeed it (i.e., the integer exponent $n$) does exist for any integer $N$ greater than $r$, as we shall see (See Section 2).\\
\indent We show that a positive integer exponent $n$ exists for each $W(r, k)$, such that for each $k$ such that the interval $[1, W(r, k)]$ has an AP of $k$ terms,
\begin{equation}
W(r, k) < r^{k^{2}}
\end{equation}
holds, if $k$ is bounded below by a certain positive real number determined by $n$ (See Theorem 2.1, Theorem 2.2, Theorem 2.3) where $W(r, k)$ is a \emph{van der Waerden number}~\cite{Graham1},~\cite{van der Waerden},~\cite{Graham and Rothschild},~\cite{Graham and Spencer},~\cite{Landman and Robertson},~\cite{Landman and Culver}, $r$ is the number of integer colorings in the interval \([1, W(r, k)]\) and $k$ is the arbitrary length of an arithmetic progression contained within this interval. The generalization of our result is as follows: Let \(W(r, k) > r, W(r, k) > k\). Then the inequality
$$
W(r, k) < r^{k^{2}},
$$
is true for all $k$ such that there is an arithmetic progression of an arbitrary \(k > 1\) number of terms in the interval $[1, W(r, k)]$, if \(k \geq \sqrt{n + 1}\). Our result as well as our approach reveals important relationships between the integers $n$, $r$, $k$ and $W(r, k)$ (See Corollary 2.1, Theorem 2.2 and Theorem 2.3, in Section 2). In fact once any $W(r, k)$ is found by some means, one easily can use $k$, $r$ and a certain real number $a(r, k)$ to derive $n$, and one can use the integers $r$, $k$, $a(r, k)$ and $n$ when known, to locate the interval in which lies $W(r, k)$ (See Theorem 2.1, Corollary 2.1, Theorem 2.2 and Theorem 2.3 in Section 2, and also Table 1 and Table 2, in Section 3).\\
\indent Previously Gowers~\cite{Gowers} had found that $W(2, k)$ is bounded above by 
\begin{equation}
2^{2^{2^{2^{2^{k + 9}}}}}.
\end{equation}
\indent The author apologizes if the techniques used in this paper are not deep enough for some readers whose field of specialization is in combinatorial number theory. The main motivation for this paper was to prove the existence of the finer upper bound in Eqtn. (1) on $W(r, k)$ than has been found previously, not in the establishment of very deep and very dazzling combinatorial results. Yet if one seeks a proof of Eqtn. (1) that does not require necessarily a deep combinatorial argument one will find it here (See Theorem 2.2). It goes without saying that less arcane, more fundamental and elegant proofs certainly are not lacking. Y. A. Khinchin~\cite{Khinchin}, derives a simple proof of the van der Waerden theorem~\cite{van der Waerden}. R. Graham and B. L. Rothschild~\cite{Graham2} found a fundamental and straightforward proof of the same theorem for 
$W(2, 3)$ that establishes an upper bound of \(W(2, 3) \leq 325\). Logician S. Shelah~\cite{Graham and Rothschild},~\cite{Graham and Spencer} designed an ingenious but fundamental and straightforward proof of the Hales--Jewett theorem~\cite{Graham and Spencer}, by defining a discrete $N$--dimensional hypercube with \emph{Shelah lines} 
$$
(x_{1}, x_{2}, \ldots, x_{26}), x_{i} \in [1, 26],
$$
where one replaces these integers with alphabet letters such that the edges and diagonals along the hypercube are word strings that receive $r$--colorings. Shelah also establishes an upper bound on $W(r, k)$ that is a WOWZER function~\cite{Graham and Rothschild},~\cite{Graham and Spencer}.\\
\indent With regard to the high levels of complexity found within some mathematical fields of specialization in general today and in some more esoteric approaches when one finds a proof to interesting conjectures in particular, Martin Davis~\cite{Davis} remarked that the mathematician ``is stymied by the abstruseness of so much of contemporary mathematics." The benefit in our simpler approach we found here to our proof that 
$$
W(r, k) < r^{k^{2}},
$$
is that our method of proof although simple (See Theorem 2.1 and Theorem 2.2) is more accessible even to those whose area of specialization is elementary number theory, to analytic number theorists, to mathematicians or computer scientists with a background in discrete mathematics and to theoretical computer scientists, all who work outside the specialized area covered by the combinatorial number theory literature. This simpler approach here to prove Eqtn. (1) does not require any intricate or arduous combinatorial argument such as double induction, as applied to $r$--colorings of the integers. Nor do we need Ackermann functions or primitive recursive functions like TOWER() or WOW(). Rather the method here depends solely \emph{on the a priori knowledge} we have already about \(W(r, k), r, k, n\) as integers with certain properties, where $n$ is a well-known integer exponent associated with any integer large enough such as any integer greater than $r$ (See Section 2), along with some simple and very basic elementary number theory and analytic number theory. Our simpler approach here not only works well. It also will lead to new knowledge and new information about the integers $n$, $r$, $k$ and $W(r, k)$ and how they relate to one another, once each new van der Waerden number $W(r, k)$ is found in the future (Please see Table 1 and Table 2, Section 3).
\section{First Result}
Suppose we are given two positive integers $N$ and \(r < N\), and then one asks us if \(N < r^{k^{2}}\) is true for positive integer $k$ we subsequently are given. Every given integer $N$ larger than some given integer $r$ has some integer expansion
\begin{equation}
N = a_{n}r^{n} + a_{n - 1}r^{n - 1} + \cdots + a_{0},
\end{equation}
into powers of $r$, where the exponent $n$ is such that $r^{n}$ divides $N$ while $r^{n + 1}$ does not, and
\begin{equation}
a_{n - 1}, \ldots, a_{0} \in [0, \ldots, r - 1],
\end{equation}
$$
a_{n} \in [1, r - 1], \: \: N \in [r^{n}, r^{n + 1}).
$$
In base $r$ arithmetic one could express $N$ as $(a_{n}a_{n - 1}\cdots a_{0})_{r}$. Yet on the other hand one just can add the terms in Eqtn. (3) in ordinary base ten arithmetic to get the original number $N$. Given any lack of more information about $N$ the expansion into powers of $r$ in Eqtn. (3) at the very least is \emph{a priori} knowledge we do have about the integer $N$. Please understand that in this paper we are not resorting to any constructive arguments that involve integer colorings of a set $[1, W(r, k)]$. We need not know beforehand whether $N$ is a newly discovered van der Waerden number with a given $r$ number of integer colorings and with an arithmetic progession of given length $k$ somewhere within integer set $[1, W(r, k)]$ or some other integer with some other property. What we do know \emph{a priori} even before we know what kind of number $N$ is in fact, is what is given in Eqtns. (3)--(4). In fact if we are given a newly discovered van der Waerden number \(N = W(r, k)\) along with \(r < N\) we always will find the exponent $n$, given $N$ and $r$ and in fact \emph{we do not need to know the actual value of $k$ to find $n$} although we will be able to verify that \(W(r, k) < r^{k^{2}}\) is true for any van der Waerden number $W(r, k)$ \emph{and for any $k$ with a certain lower bound} $\sqrt{n + 1}$, when all the values for $W(r, k)$, $n$, $r$, already are known. We stipulate all this with the stated Fact which follows:\\
\noindent \textbf{Fact:} Let $N$ be any van der Waerden number where \(N > r\) and such that the interval $[1, N]$ has some integer $r$--coloring and where this interval has an arithmetic progression of $k$ arbitrary terms. Then there always is some positive integer exponent $n$, such that the integer $r^{n}$ divides the integer $N$ while $r^{n + 1}$ does not divide $N$.\\
\indent Proceeding from this we then can proceed to show that \( N = W(r, k) < r^{k^{2}}\) will hold (See Theorem 2.2 and proof), for any $k$ that has always a certain lower bound. But first we establish one of the tools needed for the task, namely Theorem 2.1.
\begin{theorem}
Let \(r > 1, k \geq 1\). For each van der Waerden number \(W(r, k) > max(\{r, k\})\), for each integer $r$ and for each integer $k$, there exists some integer \(n \geq 1\) such that
\begin{equation}
W(r, k) < r^{n + 1}.
\end{equation}
\end{theorem}
\begin{proof}
For each positive integer $N$ and for any positive integer $r$ where \(N > r\) there exists some integer exponent $n$ that is the least integer exponent for which $r^{n}$ will divide $N$ while $r^{n + 1}$ does not, to leave some positive integer or zero remainder equal to or less than $r - 1$. Then let \(N = W(r, k)\), which certainly is possible if we are given $N$ without knowing beforehand whether or not $N$ is a van der Waerden number. For instance one can be given the integers $N$, $r$, then asked to check by means of some combinatorial argument to determine if $N$ is a van der Waerden number, after which the test will confirm that it is, so that the interval $[1, N]$ indeed is found to have both an integer $r$--coloring and an arithmetic progression of length $k$. Then substituting $W(r, k)$ for $N$, we have that $r^{n}$ divides $W(r, k)$ while $r^{n + 1}$ does not, and
\begin{equation}
b_{n}, b_{n - 1}, \ldots, b_{0} \in [1, \ldots, r - 1].
\end{equation}
That is,
\begin{equation}
1 \leq b_{n} \leq r - 1, \: b_{n}, b_{n - 1}, \ldots, b_{0} \in [0, \ldots, r - 1],
\end{equation}
so that 
\begin{equation}
W(r, k) = b_{n}r^{n} + b_{n - 1}r^{n - 1} + \cdots + b_{0},
\end{equation}
yields the base $r$ expansion of $W(r, k)$
$$
(b_{n}b_{n - 1}\cdots b_{0})_{r},
$$ 
from the powers of $r$ in Eqtn. (8). Let us focus our attention however not on the base $r$ expansion $(b_{n}b_{n - 1}\cdots b_{0})_{r}$ but instead on Eqtn. (8), where we restrict for our purposes, the expansion in Eqtn. (8) to base ten arithmetic. Then since
\begin{equation}
r > r - 1 \geq max(\{b_{n}, b_{n - 1}, \ldots, b_{0}\}),
\end{equation}
we get
\begin{equation}
W(r, k) = b_{n}r^{n} + b_{n - 1}r^{n - 1} + \cdots + b_{0} < r^{n + 1},
\end{equation}
since $n$ is defined as being the least integer exponent for which $r^{n}$ divides $W(r, k)$ but $r^{n + 1}$ does not divide $W(r, k)$, which means by definition of $n$, the positive integer exponent $n$ is an exponent of $r$ for which \(r^{n + 1} > W(r, k)\) is true.
\end{proof}
\textbf{Remark} The following result one can derive automatically as a corollary to Theorem 2.1. Since the proof is straightforward we leave it to the reader to convince himself or herself as to the truthfulness of the result.\\
\begin{corollary}
$$
n > \frac{\log W(r, k)}{\log r} - 1.
$$
\end{corollary}
\indent The expansion of $W(r, k)$ into powers of $k$ also is possible. Nevertheless \(W(r, k) < r^{n + 1}\) remains true when we do arithmetic in base ten even when we expand $W(r, k)$ into powers of $k$. Let positive integer exponent $m$ be such that $k^{m}$ divides each van der Waerden number $W(r, k)$ but $k^{m + 1}$ does not divide $W(r, k)$, where
$$
W(r, k) = c_{m}k^{m} + c_{m - 1}k^{m - 1} + \cdots + c_{0},
$$
and where \(c_{m}, c_{m - 1}, \ldots, c_{0} \in [0, k - 1]\), \(1 \leq c_{m} < k\). Then in ordinary base ten arithmetic we get and by using Theorem 2.1,
\begin{eqnarray}
c_{m}k^{m} + c_{m - 1}k^{m - 1} + \cdots + c_{0}&=&b_{n}r^{n} + b_{n - 1}r^{n - 1} + \cdots + b_{0}\nonumber\\
                                                &=&W(r, k)\nonumber\\
                                                &<&r^{n + 1}.
\end{eqnarray}
Put into words, the expansion into powers of $k$ and the expansion into powers of $r$ are two different expansions of the same integer (when we restrict ourselves to base ten arithmetic), namely the van der Waerden number $W(r, k)$, where this number is smaller than $r^{n + 1}$ but also smaller than $k^{m + 1}$. This remains true whether \(r < k\) or \(r > k\). This means both the inequalities \(W(r, k) < k^{m + 1}\) and \(W(r, k) < r^{n + 1}\) are true whether \(r < k\) or \(r > k\), and regardless of which of these expansions we give to $W(r, k)$. In fact the result we really have here is that
$$
W(r, k) < min(\{r^{n + 1}, k^{m + 1}\}).
$$
\indent When \(r = k\) we have also that \(m = n\), \(r^{n + 1} = k^{m + 1}\).
\subsection{Why \(W(r, k) < r^{k^{2}}\) is true for each $k$ provided $k$ is equal to or greater than a certain lower Bound}
Suppose \(k = k_{0}\) remains always some same nonarbitrary fixed constant as \(r \rightarrow \infty\). Then no interval $[1, W(r, k_{0})]$ will contain an AP for an arbitrary $k$ number of terms. Therefore from here on we consider what happens while $r$ is any integer coloring and while $k$ remains arbitrary.
\begin{theorem}
Let $W(r, k)$ be any van der Waerden number, where \(r > 1, n \geq 1\), \(r < W(r, k), k < W(r, k)\), and where $n$ is as defined in Eqtn. (6). Then 
\begin{equation}
W(r, k) < r^{n + 1} \leq r^{k^{2}}
\end{equation}
is true provided that for any integer \(k > 1\) such that the interval $[1, W(r, k)]$ has an arithmetic progression of an arbitrary $k$ terms from a monochromatic $r$--coloring among some of its integer elements, \(k \geq \sqrt{n + 1}\). 
\end{theorem}
\begin {proof}
First we prove the theorem is true for all \(k > \sqrt{n + 1}\). Let
\begin{equation}
W(r, k) = b_{n}r^{n} + b_{n - 1}r^{n - 1} + \cdots + b_{0},
\end{equation}
be the expansion of $W(r, k)$ into powers of $r$. Then by recourse to Theorem 2,1,
\begin{eqnarray}
r^{n}&\leq&W(r, k) = b_{n}r^{n} + b_{n - 1}r^{n - 1} + \cdots + b_{0}\\
     &<&r^{n + 1}\nonumber \\
     &\Longrightarrow&W(r, k) < r^{n + 1} < r^{k^{2}}
\end{eqnarray}
is true for all \(k > \sqrt{n + 1}\).\\
\indent Now we prove the theorem is true for all \(k \geq \sqrt{n + 1}\). Suppose that \(t = \sqrt{n + 1}\) is true actually, for some integer \(t > 1\). In fact $t$ actually is a positive integer \(t = 2\), for the van der Waerden numbers $W(2,3)$, $W(3, 3)$ and $W(4, 3)$ (See Table 1). Then $1$ is a quadratic residue modulo $n$, since \(\gcd(t, n) = 1\) and
$$
t^{2} \equiv (n + 1) \: \equiv \: 1 \: mod \: n,
$$ 
so that we get
$$
r^{n} \leq W(r, k) < r^{n + 1} \leq r^{k^{2}}
$$
is true also for all \(k \geq \sqrt{n + 1}\).\\
\indent Therefore combining these two results, we have proved that \(r^{n} \leq W(r, k) < r^{n + 1} \leq r^{k^{2}}\) is true for all \(k > 1\) for which the interval $[1, W(r, k)]$ has an arithmetic progression of $k$ arbitrary terms, provided that \(k \in [\sqrt{n + 1}, \infty)\) is true. 
\end{proof}
\textbf{Remark:} The reader should be aware that since \(r^{n} \leq W(r,k) < r^{n +1}\) is true always for each van der Waerden number, the integer square exponent $k^{2}$ cannot lie between the two integer exponents $n$ and $n + 1$, since on $\mathbb{R}$ we have \(diam([n, n + 1]) = 1\), which means obviously there are no integers that lie between the two endpoints of the set $[n, n + 1]$. It follows from this automatically that if \(W(r,k) < r^{k^{2}}\) is to be true always it must be that for any $k$ such that $[1, W(r, k)]$ contains an AP of $k$ terms, \(k \geq \sqrt{n + 1}\) is true always, since \(W(r, k) < r^{n + 1}\) and there are no integers in the interval $[n, n + 1]$ other than $n$ and $n + 1$. Conversely, if \(k \geq \sqrt{n + 1}\) is true always for all $k$ where the integer $k$ is the length of some AP in the interval $[1, W(r, k)]$, then \(W(r, k) < r^{n + 1} \leq r^{k^{2}}\). The Author has confirmed this for all currently known van der Waerden numbers $W(r, k)$ (See Tables, Section 3).\\   
The value $\sqrt{n + 1}$ is a lower bound on the integers $k$ such that the inequality \(W(r, k) < r^{k^{2}}\) is true, but so far we have not shown it is the greatest lower bound. Let $k$, $k^{\prime}$ be two distinct positive integers with \(k > k^{\prime}\). In general,
$$
W(r, 2) < W(r, 3) < \cdots < W(r, k^{\prime}) < W(r, k)
$$
holds for all \(k > k^{\prime}\), a fact which we use in the proof to Corollary 2.2. With the next Corollary we provide a stronger result from that we found with Theorem 2.2. Moreover we also shall show the greatest lower bound for $k$ must be $\sqrt{n + 1}$. 
\begin{corollary}
Let \(\sqrt{n + 1} \in \mathbb{R} - \mathbb{Q}\). Then
\begin{equation}
W(r, k) < r^{n + 1} < r^{k^{2}},
\end{equation}
is true for any $k$ provided that \(k \in (\sqrt{n + 1}, \infty)\), where the integer interval $[1, W(r, k)]$ has an AP of an arbitrary number of \(k > \sqrt{n + 1}\) terms. Moreover for each van der Waerden number $W(r, k)$, let \(\delta(r, k) \in [n, n + 1) \subset \mathbb{R}^{+}\) be that positive real exponent for which \(W(r, k) = r^{\delta(r, k)}\). Then the integer $k$ has the greatest real positive lower bound $\sqrt{n + 1}$ on the interval $[\sqrt{n + 1}, \infty)$, such that \(W(r, k) < r^{k^{2}}\) is true for integer \(k \in (\sqrt{n + 1}, \infty)\). \\
\indent On the other hand if $\sqrt{n + 1}$ actually is integer instead of being found inside \(\mathbb{R} - \mathbb{Q}\), then \(W(r, k) < r^{n + 1} \leq r^{k^{2}}\) is true for any arbitrary integer \(k \in [\sqrt{n + 1}, \infty)\), where $\sqrt{n + 1}$ actually is the greatest integer lower bound.   
\end{corollary}
\begin{proof}
\indent Let $k^{\prime}$ be any positive integer smaller than $k$, but also such that \(k^{\prime} \leq \sqrt{n}\) \(\Longrightarrow k^{\prime 2} \leq n\) and such that $W(r, k^{\prime})$ is any van der Waerden number smaller than the van der Waerden number $W(r, k)$, where the integer interval $[1, W(r, k^{\prime})]$ has an AP of some arbitrary $k^{\prime}$ number of terms. Observe that if \(k^{\prime} \leq \sqrt{n}\) then \(k^{\prime 2} \not \in (n, n + 1)\) because the open interval $(n, n + 1)$ on $\mathbb{R}$ contains no integer. Let $n^{\prime}$ be that positive integer exponent for which $r^{n^{\prime}}$ divides $W(r, k^{\prime})$ but $r^{n^{\prime} + 1}$ does not divide $W(r, k^{\prime})$.\\
\indent First we consider the given case \(\sqrt{n + 1} \in \mathbb{R} - \mathbb{Q}\). \\
Then from Theorem 2.1 we have also for any arbitrary integer \(k > k^{\prime}\) and for any van der Waerden number \(W(r, k) > W(r, k^{\prime})\),
$$
r^{n^{\prime}} \leq W(r, k^{\prime}) \leq W(r, \lfloor \sqrt{n + 1} \rfloor) < W(r, k) < r^{n + 1} < r^{k^{2}},
$$
certainly is true for arbitrary integer $k$ where \(k \in (\sqrt{n + 1}, \infty)\), and where the integer interval $[1, W(r, k)]$ which is larger than the interval $[1, W(r, k^{\prime})]$, has an AP for some arbitrary number of $k$ terms.\\
\indent Now since for each $W(r, k)$ we have 
$$
W(r, k) = r^{\delta(r, k)} \in [r^{n}, r^{n + 1}),
$$
then \(n \leq \delta(r, k) < n + 1\) must be true always for each $W(r, k)$, because the expansion of each $W(r, k)$ into powers of $r$ cannot be smaller than $r^{n}$ and it cannot be as large as is $r^{n + 1}$. That is,
$$
r^{n} \leq W(r, k) = b_{n}r^{n} + b_{n - 1}r^{n - 1} + \cdots + b_{0} < r^{n + 1},
$$
implies \(\delta(r, k) \in [n, n + 1)\), where $n$ is the least positive integer exponent for which $r^{n}$ divides $W(r, k)$ but $r^{n + 1}$ does not. Furthermore this positive real number exponent $\delta(r, k)$ cannot be an integer larger than $n$ because the interval $[n, n + 1)$ on $\mathbb{R}$ cannot contain any integer other than $n$ (See Table 3). All this then means that
\begin{eqnarray}
r^{n} \leq W(r, k)&=              &r^{\delta(r, k)} < r^{\lceil \delta(r, k) \rceil} \leq r^{n + 1} < r^{k^{2}}\nonumber\\
                  &\Longrightarrow&k^{2} \in [\lceil \delta(r, k) \rceil, \infty) = [n + 1, \infty)\nonumber\\
                  &\Longrightarrow&k \in [\sqrt{\lceil \delta(r, k) \rceil}, \infty) = [\sqrt{n + 1}, \infty),\nonumber
\end{eqnarray}
where we have used the fact that \(n \leq \delta(r, k) < n + 1\) \(\Longrightarrow \lceil \delta(r, k)\rceil = n + 1\). \\
\indent Combining all these results we see that 
\begin{enumerate}
\item \(k^{\prime} \leq \sqrt{n}\) for any arbitrary positive integer \(k^{\prime} < k\), such that \(W(r, k^{\prime}) < W(r, k)\),\\
\item  The interval $[n, n + 1)$ contains always $\delta(r, k)$ the positive real exponent such that \(W(r, k) = r^{\delta(r, k)} < r^{n + 1}\), and the open interval $(n, n + 1)$ cannot contain the integer square $k^{\prime 2}$.\\
\item Since the interval $[n, n + 1)$ can contain no integer other than $n$ and since \(n \leq \delta(r, k) < n + 1\), this implies \([\lceil\delta(r, k)\rceil, \infty) = [n + 1, \infty)\).\\
\item \(k^{\prime} \leq \sqrt{n}\), \(k^{\prime} < k\) and \(W(r, k^{\prime}) \leq W(r, \lfloor \sqrt{n + 1}\rfloor) < W(r, k)\) \(\Longrightarrow k \in (\sqrt{\lceil \delta(r, k)\rceil}, \infty)\), where \(\lceil \delta(r, k)\rceil = n + 1\) and $\sqrt{n + 1}$ is the greatest positive real lower bound on any integer $k$ within the set $[\sqrt{n + 1}, \infty)$.
\end{enumerate}
\indent Now we consider the case when $\sqrt{n + 1}$ actually is an integer. Then, since \(\sqrt{n + 1} = \lfloor \sqrt{n + 1}\rfloor = \lceil \sqrt{n + 1}\rceil \), 
\begin{eqnarray}
r^{n^{\prime}}&\leq&W(r, k^{\prime}) \leq W(r, \lfloor \sqrt{n + 1}\rfloor)\nonumber\\
              &\leq&W(r, k) < r^{n + 1} \leq r^{k^{2}}\nonumber
\end{eqnarray}
also is true for any arbitrary integer \(k \in [\sqrt{n + 1}, \infty)\), where $\sqrt{n + 1}$ is the greatest integer lower bound on this interval.
\end{proof}
The real positive exponent $\delta(r, k)$, which for each $W(r, k)$ has the form
$$
\delta(r, k) = \frac{\log W(r, k)}{\log r},
$$
is easy to compute where logs can be taken either to the base $e$ or $r$ (See Table 3, Section 3) for all the known van der Waerden numbers of the type $W(r, k)$ such as those known van der Waerden numbers that appear in Table 1 and in Table 2, and the real exponent $\delta(r, k)$ can be computed easily (along with the exponent $n$) in the future, whenever any van der Waerden numbers $W(r, k)$ that currently are unknown are found at a future date.\\
\indent As one can see from Table 1, there indeed are cases for which \(k = n\) holds as $W(r, n)$ gets larger (See Corollary 2.3). In this case we also have a necessary condition for which \(W(r, k) < r^{k^{2}}\) holds for all \(k > \sqrt{n + 1}\), which in fact is easy to demonstrate.
\begin{corollary}
Whenever \(k \geq n\) for \(n > 1\), \(W(r, k) < r^{k^{2}}\) is true for all \(k > \sqrt{n + 1}\).
\end{corollary}
\begin{proof}
If \(k = n\) then 
\begin{eqnarray}
r^{n} \leq W(r, n)&<   &r^{n + 1}\nonumber\\
                  &<   &r^{n^{2}} = r^{k^{2}},\nonumber
\end{eqnarray}
where clearly we have that \(n + 1 < n^{2} = k^{2}\) \(\Longrightarrow \sqrt{n + 1} < k\) is true for all \(k = n\) and for \(n > 1\).\\
\indent Now suppose \(k > n > 1\). Then \(k > n > \sqrt{n + 1}\) which means \(W(r, k) < r^{n + 1} < r^{k^{2}}\) still is true for all integer \(k \in (\sqrt{n + 1}, \infty)\), since \(W(r, k) = r^{\delta(r, k)}\) always for some \(\delta(r, k) \in [n, n + 1) \subset [n, k^{2})\).
\end{proof}
\indent With Theorem 2.2, Corollary 2.2 and Corollary 2.3 we have proved an important and highly relevant result when it comes to finding bounds on any van der Waerden number $W(r, k)$. That is, whether or not \(W(r, k) < r^{k^{2}}\) is true for each $k$ depends upon \(k \in [\sqrt{n + 1}, \infty)\), where $n$ always is that special integer exponent for which $r^{n}$ divides $W(r, k)$ but $r^{n + 1}$ does not divide $W(r, k)$, and so that $W(r, k)$ can be expanded into powers of $r$ where the exponent powers of $r$ in the expansion of each $W(r, k)$ are no greater than $n$. Each van der Waerden number $W(r, k)$ then lies in the interval $[r^{n}, r^{n + 1}]$ while each arbitrary $k$ lies on the interval $[\sqrt{n + 1}, \infty)$. All the values of $k$ that are possible are bounded below by $\sqrt{n + 1}$. This makes sense, since if \(n \in [1, \infty)\) is to be bounded below by one, then for $k$ to be a nontrivial positive intger it cannot be any smaller than \(\sqrt{1 + 1} = \sqrt{2}\).\\
\indent Then given $r$ and $k$, a computational number theorist can estimate the size of $n$, to locate the possible interval in which each $W(r, k)$ lies. For instance suppose it is known that, for some unknown $W(r, k)$,
$$
W(r, k) \in [a, b], \: a, b \in \mathbb{N}.
$$
Then \(\frac{a}{r^{n}} \leq \frac{W(r, k)}{r^{n}} \leq \frac{b}{r^{n}}\).\\
\indent R. Graham, B. Rothschild and J. Spencer have conjectured~\cite{Graham and Rothschild},~\cite{Graham and Spencer}, that
$$
W(2, k) < 2^{k^{2}}.
$$
The truth of their Conjecture follows at once as a special case of Theorem 2.1 and Theorem 2.2, when \(k \geq \sqrt{n + 1}, r = 2\) (See the two Tables in Section 3).\\
\indent The following theorem establishes further relationships between the numbers $n$, $r$, $k$ and $W(r, k)$.
\begin{theorem}
Let \(W(r, k) > max(\{r, k\})\), \(r > 1, k > 1, n \geq 1\) and suppose Theorem 2.1 and Theorem 2.2 hold. Then the following five conditions also hold:
\begin{enumerate}
\item \(n > \frac{\log k}{\log r} - 1\),\\
\item \(\sqrt[n + 1]{k} = O(r)\),
\item For each triplet \(n, r, k\) there exists \(a(r, k) \in \mathbb{R}\), such that
\begin{equation}
n = a(r, k)\frac{\log k}{\log r} - 1,
\end{equation}
if and only if
\begin{equation}
a(r, k)= (n + 1)\frac{\log r}{\log k}
\end{equation}
where, if \(r < k\) then \(a(r, k) < n + 1\), if \(r = k\) then \(a(r, k) = n + 1\) and if \(r > k\) then \(a(r, k) > n + 1\).\\
\item For each triplet \(n, r, k\) there exists \(a(r, k) \in \mathbb{R}\), such that
\begin{equation}
k = r^{\frac{n + 1}{a(r, k)}}
\end{equation}
if and only if
\begin{equation}
a(r, k)= (n + 1)\frac{\log r}{\log k}
\end{equation}
where, if \(r < k\) then \(a(r, k) < n + 1\), if \(r = k\) then \(a(r, k) = n + 1\) and if \(r > k\) then \(a(r, k) > n + 1\).\\
\item For large $n$, \(\frac{W(r, k)}{r^{n}} = O(r)\).
\end{enumerate}
\end{theorem}
\begin{proof}
To prove Condition 1, we use Theorem 2.1 and since 
\begin{equation}
W(r, k) > max(\{r, k\}),
\end{equation}
implies \(W(r, k) > k\),
\begin{eqnarray}
& &k < W(r, k) = b_{n}r^{n} + b_{n - 1}r^{n - 1} + \cdots + b_{0} < r^{n + 1}\\
&\Longrightarrow&\log k < (n + 1)\log r\nonumber \\
&\Longrightarrow&\frac{\log k}{\log r} - 1 < n.
\end{eqnarray}
To prove Condition 2 we use the fact that \(k < W(r, k) < r^{n + 1} \Longrightarrow k < r^{n + 1}\),
\begin{eqnarray}
k < r^{n + 1}&\Longrightarrow&\lim_{n \rightarrow \infty}\sqrt[n + 1]{k} \leq  \lim_{n \rightarrow \infty} r = r\nonumber\\
             &\Longrightarrow&\sqrt[n + 1]{k} = O(r).
\end{eqnarray}
Next we prove Condition 3 and Condition 4. 
For Condition 3,
\begin{eqnarray}
a(r, k)&=&(n + 1)\frac{\log r}{\log k} \Longrightarrow a(r, k)\frac{\log k}{\log r} = (n + 1)\nonumber\\
       &\Longrightarrow&a(r, k)\frac{\log k}{\log r} - 1 = n.\nonumber
\end{eqnarray}
Conversely
\begin{eqnarray}
n = a(r, k)\frac{\log k}{\log r} - 1&\Longrightarrow&(n + 1)\log r = a(r, k)\log k\nonumber \\
                                    &\Longrightarrow&(n + 1)\frac{\log r}{\log k} = a(r, k).\nonumber
\end{eqnarray}
For Condition 4,
\begin{eqnarray}
a(r, k) = (n + 1)\frac{\log r}{\log k}&\Longrightarrow&\log k = \frac{n + 1}{a(r, k)}\log r\nonumber \\
                                      &\Longrightarrow&k = r^{\frac{n + 1}{a(r, k)}}.\nonumber
\end{eqnarray}
Conversely
\begin{eqnarray}
k = r^{\frac{n + 1}{a(r, k)}}&\Longrightarrow&\log k = \frac{n + 1}{a(r, k)}\log r\nonumber \\
                             &\Longrightarrow&a(r, k) = (n + 1)\frac{\log r}{\log k}.\nonumber
\end{eqnarray}
Finally for Condition 5 and by using results from Theorem 2.1 and Theorem 2.2, 
\begin{eqnarray}
&               &r^{n} \leq W(r, k) < r^{n + 1}\nonumber\\
&\Longrightarrow&1 \leq \frac{W(r, k)}{r^{n}} < r\nonumber\\
&\Longrightarrow&\lim_{n \rightarrow \infty}1 \leq \lim_{n \rightarrow \infty}\frac{W(r, k)}{r^{n}} \leq \lim_{n \rightarrow \infty} r\nonumber\\
&\Longrightarrow&\frac{W(r, k)}{r^{n}} = O(r).\nonumber
\end{eqnarray}
This completes the proof.
\end{proof}
\textbf{Remark} Just because the number $a(r, k)$ that appears in Condition 3 and in Condition 4 depends upon the values of $r$ and $k$, such dependence upon these arguments of $W(r, k)$ does not mean that this real number $a(r, k)$ does not exist for each and every allowable choice of $r$ and $k$ and for each and every van der Waerden number (See Table 1). The commutative field $\mathbb{R}$ has every nonzero element in it being a unit and it is closed under addition and $\mathbb{R} - \{0\}$ is closed under multiplication and so the field of real numbers contains all the required field elements \(n + 1, k, r, \log k, \log r, a(r, k)\) with the required field operations \(+, \times\) holding between them for both Condition 3 and Condition 4 to hold. In fact we do have a priori knowledge that each real number $a(r, k)$ must exist in $\mathbb{R}$ for each value of $n$, for each value of $W(r, k)$, for each value of $r$ and for each value of $k$ (Table 1).
\begin{corollary}
$$
r = k^{\frac{a(r, k)}{n + 1}} = k^{\frac{\log r}{\log k}},
$$
$$
r^{n} = k^{\frac{a(r, k)n}{n + 1}} = k^{\frac{n\log r}{\log k}}.
$$
\end{corollary}
\begin{proof}
This follows from applying Condition 3 and Condition 4 in Theorem 2.3 and by solving for $r$ in Eqtn. (17) and by substituting \((n + 1)\frac{\log r}{\log k}\) for $a(r, k)$ from Eqtn. (18), Condition 4, Theorem 2.3. 
\end{proof}
\subsection{The Case \(r = k\)}
\begin{corollary}
Suppose Theorem 2.1, Theorem 2.2 and Theorem 2.3 hold. Let
\begin{equation}
\log r > \frac{a(r, k)\log(n+ 1)}{2(n + 1)}.
\end{equation}
Then \(k \geq \sqrt{n + 1}\). Moreover if \(r = k\) then \(r = k \geq \sqrt{n + 1}\).
\end{corollary}
\begin{proof}
From Condition 3 and Condition 4 in Theorem 2.3,
\begin{eqnarray}
& &\log r > \frac{a(r, k)\log(n+ 1)}{2(n + 1)}\\
&\Longrightarrow&\log k = \frac{(n + 1)\log r}{a(r, k)} > \frac{1}{2}\log(n + 1)\nonumber \\
&\Longrightarrow&k = r^{\frac{n + 1}{a(r, k)}} \geq \sqrt{n + 1}.
\end{eqnarray}
Finally suppose \(r = k\). Then from Condition 3 and Condition 4 in Theorem 2.3,
\begin{eqnarray}
\log r = \log k&>&\frac{a(k, k)\log(n+ 1)}{2(n + 1)} = \frac{(n + 1)\log(n+ 1)}{2(n + 1)}\\
               &=&\frac{1}{2}\log(n + 1)\nonumber \\
               &\Longrightarrow&r = k \geq \sqrt{n + 1}.
\end{eqnarray}
\end{proof}
\begin{corollary}
Suppose \(n > k\). Then if \(\sqrt{n + 1} \leq k\) then \(W(r, k) < r^{k^{2}}\). On the other hand if \(n \leq k\) then still, \(W(r, k) < r^{k^{2}}\).
\end{corollary}
\begin{proof}
If \(n > k\) but \(\sqrt{n + 1} \leq k\) then by applying Theorem 2.1 and Theorem 2.2, \(W(r, k) < r^{n + 1} \leq r^{k^{2}}\). On the other hand if \(n \leq k\) then it follows that \(\sqrt{n + 1} \leq k\) as well, meaning by applying Theorem 2.1 and Theorem 2.2, that \(W(r, k) < r^{n + 1} < r^{k^{2}}\).
\end{proof}
\section{Computational Results}
We apply the results from the previous Sections for known van der Waerden numbers in Table 1 and in Table 2. Logarithms are taken to the base \(e = 2.718\cdots\) In the Table all the values for $k$ lie within the interval \((\sqrt{n + 1}, n + 1)\).
\begin{center}
\begin{tabular}{|l      |c      |c                  |c      |c            |c              |c               |c             |c              |r|}
\hline
                $r$  &  $k$  &  $\sqrt{n + 1}$  &   $n$  &  $\log r$  &   $\log k$    &   $a(r, k)$    &  $W(r, k)$   &   $r^{n + 1}$  &  $r^{k^{2}}$ \\   
\hline
                $2$  &  $3$  &  $2$             &   $3$  &  $0.6931$  &   $1.0986$    &   $2.5235$     &  $9$         &   $2^{4}$      &  $2^{9}$  \\
 
                $2$  &  $4$  &  $2.449\ldots$   &   $5$  &  $0.6931$  &   $1.3862$    &   $3$          &  $35$        &   $2^{6}$      &  $2^{16}$  \\
                
                $2$  &  $5$  &  $2.828\ldots$   &   $7$  &  $0.6931$  &   $1.6094$    &   $3.4452$     &  $178$       &   $2^{8}$      &  $2^{25}$ \\ 

                $2$  &  $6$  &  $3.316\ldots$   &   $10$ &  $0.6931$  &   $1.7917$    &   $4.2552$     &  $1132$      &   $2^{11}$     &  $2^{36}$  \\
                 
                $3$  &  $3$  &  $2$             &   $3$  & $1.0986$   &   $1.0986$    &   $4$          &  $27$        &   $3^{4}$      &  $3^{9}$  \\
                
                $3$  &  $4$  &  $2.449\ldots$   &   $5$  &  $1.0986$  &   $1.3862$    &   $4.7551$     &  $293$       &   $3^{6}$      &  $3^{16}$  \\

                $4$  &  $3$  &  $2$             &   $3$  &  $1.3862$  &   $1.0986$    &   $5.0471$     &  $76$        &   $4^{4}$      &  $4^{9}$  \\
  
\hline
\end{tabular}
\end{center}
\begin{center}
Table 1.
\end{center}
\begin{center}
\begin{tabular}{|l      |c      |c      |c                       |c                                                                                              |r|}
\hline
                $r$  &  $k$  &  $n$   &  \(W(r, k) = N\)     &    \(N =\)                                                                                   \\   
\hline
                $2$  &  $3$  &  $3$   &  \(W(2, 3) = 9\)     &    \(9 = 1\cdot 2^{3} + 1\cdot 2^{0}\)                                                  \\
 
                $2$  &  $4$  &  $5$   &  \(W(2, 4) = 35\)    &    \(35 = 1\cdot 2^{5} + 1\cdot 2^{1} + 1\cdot 2^{0}\)                                 \\
                
                $2$  &  $5$  &  $7$   &  \(W(2, 5) = 178\)   &    \(178 = 1\cdot 2^{7} + 1\cdot 2^{5} + 1\cdot 2^{4} + 1\cdot 2^{1}\)                    \\ 

                $2$  &  $6$  &  $10$  &  \(W(2, 6) = 1132\)  &    \(1132 = 1\cdot 2^{10} + 1\cdot 2^{6} + 1\cdot 2^{5} + 1\cdot 2^{3} + 1\cdot 2^{2}\)   \\
                 
                $3$  &  $3$  &  $3$   &  \(W(3, 3) = 27\)    &    \(27 = 1\cdot 3^{3}\)                                                                  \\
                
                $3$  &  $4$  &  $5$   &  \(W(3, 4) = 293\)   &    \(293 = 1\cdot 3^{5} + 1\cdot 3^{3} + 2\cdot 3^{2} + 1\cdot 3^{1} + 2\cdot 3^{0}\)   \\

                $4$  &  $3$  &  $3$   &  \(W(4, 3) = 76\)    &    \(76 = 1\cdot 4^{3} + 3\cdot 4^{1}\)                                                  \\
  
\hline
\end{tabular}
\end{center}
\begin{center}
Table 2.
\end{center}
\newpage
\begin{center}
\begin{tabular}{|l      |c      |c          |c                              |c                         |c                                       |c                                        |r|}
\hline
               $r$  &  $k$  &  $n$   &      $\delta(r, k)$           &      \(W(r, k) = N\)     &      \(N = r^{\delta(r, k)}\)        &         \(\delta(r, k) \in [n, n + 1)\)  \\   
\hline
                $2$  &  $3$  &  $3$   &      $3.17010\ldots$          &      \(W(2, 3) = 9\)     &       \(9 = 2^{3.17010\ldots}\)     &         \(3.17010\ldots \in [3, 4)\)          \\
 
                $2$  &  $4$  &  $5$   &      $5.12693\ldots$          &      \(W(2, 4) = 35\)    &       \(35 = 2^{5.12693\ldots}\)    &         \(5.12693\ldots \in [5, 6)\)          \\
                
                $2$  &  $5$  &  $7$   &      $7.47623\ldots$          &      \(W(2, 5) = 178\)   &       \(178 = 2^{7.47623\ldots}\)   &         \(7.47623\ldots \in [7, 8)\)          \\ 

                $2$  &  $6$  &  $10$  &      $10.14534\ldots$         &      \(W(2, 6) = 1132\)  &       \(1132 = 2^{10.14534\ldots}\) &         \(10.14534\ldots \in [10, 11)\)      \\
                 
                $3$  &  $3$  &  $3$   &      $3.00002\ldots$          &      \(W(3, 3) = 27\)    &       \(27 = 3^{3.00002\ldots}\)    &         \(3.00002\ldots \in [3, 4)\)        \\
                
                $3$  &  $4$  &  $5$   &      $5.17037\ldots$          &      \(W(3, 4) = 293\)   &       \(293 = 3^{5.17037\ldots}\)   &         \(5.17037\ldots \in [5, 6)\)     \\

                $4$  &  $3$  &  $3$   &      $3.12417\ldots$          &      \(W(4, 3) = 76\)    &       \(76 = 4^{3.12417\ldots}\)    &         \(3.12417\ldots \in [3, 4)\)     \\
  
\hline
\end{tabular}
\end{center}
\begin{center}
Table 3.
\end{center}
\subsection{Location of $\log_{r}W(r, k)$ on $\mathbb{R}^{+}$}
What happens when each \(\log W(r, k) = \delta(r, k)\) in Table 3 is taken instead to the base $r$? Our results with Theorem 2.1, Theorem 2.2, Corollary 2.2 and Corollary 2.3, then enable us to use far less computer time in determining the right subinterval 
\begin{equation}
[r^{n}, r^{n + 1}),
\end{equation}
on the real line that contains $W(r, k)$ for any fixed $r$, as 
\begin{equation}
W(r, k) \rightarrow \infty, \: k \rightarrow \infty,
\end{equation}
by determining instead and first for any given $n$, the right subinterval 
\begin{equation}
[n, n + 1) \subset [1, n + 1),
\end{equation}
that contains the logarithm to the base $r$ of $W(r, k)$, since 
\begin{equation}
r^{n} \leq W(r, k) < r^{n + 1} \Longrightarrow n \leq \log_{r}W(r, k) < n + 1,
\end{equation}
using the fact that the logarithm of $r$ to the base $r$ is one. One ought to note that both the positive exponents $n$ and $\delta(r, k)$ increase much more slowly then does $W(r, k)$ and that both of these are $O(n + 1)$. Since there is no integer within the open interval $(n, n + 1)$, the integer square $k^{2}$ cannot lie in this interval. Furthermore as we have seen in the Proof to Corollary 2.2, if we are given any arbitrary integer $k^{\prime}$ but where 
\begin{equation}
k^{\prime} \leq \sqrt{n}, \: k^{\prime} < k
\end{equation}
hold, there will be some van der Waerden number $W(r, k^{\prime})$ smaller than $W(r, k)$, which means there still has to be some arbitrary integer $k$ equal to or greater than $\sqrt{n + 1}$ where the exponent $n$ is as we have defined it in Theorem 2.1. Thus we have, following from Corollary 2.2,
\begin{equation}
W(r, k^{\prime}) \in [r^{n^{\prime}}, r^{n^{\prime} + 1}), 
\end{equation}
for all $k^{\prime}$ in $[1, \sqrt{n + 1})$,and
\begin{equation}
W(r, k) \in [r^{n}, r^{n + 1}), W(r, k) > W(r, k^{\prime}),
\end{equation}
for all $k$ in $[\sqrt{n + 1}, \infty)$.
\section{Boundedness for $W(2, k)$}
In this Section we use the definition~\cite{Graham and Spencer},
$$
EXPONENT(k^{2}) = 2^{k^{2}}.
$$
\begin{theorem}
Let Theorem 2.1 and Theorem 2.2 hold, such that 
\begin{equation}
W(r, k) < r^{n + 1} \leq r^{k^{2}}.
\end{equation}
Then \(W(2, k) \leq EXPONENT(k^{2})\).
\end{theorem}
\begin{proof}
This follows automatically by setting \(r = 2\), since
\begin{equation}
W(2, k) < 2^{n + 1} \leq EXPONENT(k^{2}).
\end{equation}
\end{proof}
We derive another result which holds for each van der Waerden number $W(2, k)$, when \(r = 2\) is for a 2--coloring. 
\begin{theorem}
Suppose both of the following Conditions are met for each $W(2, k)$:
\begin{enumerate}
\item \(k^{2}, k \in [1, W(2, k)]\),\\
\item There exists some integer \(n \in [1, k^{2} - 1) \subset [1, W(2, k)]\), where $n$ is the least integer exponent for which $2^{n}$ divides $W(2, k)$.
\end{enumerate}
Then \(W(2, k) < 2^{k^{2}}\). 
\end{theorem}
\begin{proof}
By Condition 1, \(k^{2} \in [1, W(2, k)]\). By Condition 2, \(n \leq k^{2} - 1\). Therefore by Theorem 2.1 when \(r = 2\), by Condition 1 and by Condition 2,
\begin{equation}
W(2, k) < 2^{n + 1} \leq 2^{k^{2}}.
\end{equation}
\end{proof}
\subsection{\emph{When \(k - 1 = p\) is prime}}
The upper bounds we found in the previous Sections help us to find an integer interval on the real line in which one can find van der Waerden numbers when \(r = 2, k - 1 = p\), where $p$ is a prime.
\begin{theorem}
Let Theorem 4.2 hold and let \(k = p + 1\), where $p$ is any prime such that
\begin{equation}
(\sqrt{n + 1}) - 1 \leq p.
\end{equation}
Furthermore let Theorem 2.1 and Theorem 2.2 hold. Then 
\begin{equation}
p2^{p} < W(2, p + 1) < 2^{(p + 1)^{2}}.
\end{equation}
\end{theorem}
\begin{proof}
The lower bound in Eqtn. (40) was established already by Berlekamp~\cite{Berlekamp}. So if $p$ is bounded below as given in Eqtn. (38) then
\begin{eqnarray}
p2^{p}&<&W(2, p + 1) = b_{n}2^{n} + b_{n - 1}2^{n - 1} + \cdots + b_{0}\\
      &<&2^{n + 1} \leq 2^{(p + 1)^{2}}\nonumber \\
      &\Longrightarrow&p2^{p} < W(2, p + 1) < 2^{(p + 1)^{2}}.
\end{eqnarray}
\end{proof}
\section{Concluding Remarks}
Theorem 2.2 establishes that \(W(r, k) < r^{k^{2}}\) is true (See Table 1) if integer $k$ is bounded below by $\sqrt{n + 1}$, where the integer $n$ was defined in the proof to Theorem 2.1. We determined this result without finding, by an inductive or constructive argument or through some algorithmic process perhaps, any particular arithmetic progression of length $k$ for a monochromatic integer $r$--coloring inside some integer interval $[1, N]$ where $N$ turns out to equal $W(r, k)$. \\
\indent There might be some specialists, at least some combinatorial number theorists included perhaps, who will chafe at the simpler approach we have dared so boldly to choose, to show there exists a condition for which \(W(r, k) < r^{k^{2}}\) is true, namely if, for all $k$ such that the interval $k$ has an AP of $k$ terms, we have \(k \in [\sqrt{n + 1}, \infty)\), where $n$ is that integer exponent for which \(r^{n} \leq W(r, k) < r^{n + 1}\) is true for each van der Waerden number $W(r, k)$. Yet long ago in Western civilization's past one did say ``Facts do not cease to exist because they are ignored."~\footnote{A quote from a work by British intellectual and writer Aldous Huxley (1894--1963).} \\
\indent Therefore for our approach in this paper not to work, at least one of the following statements must be true:
\begin{enumerate}
\item There does not exist an expansion into some powers of $r$ for each and every van der Waerden number $W(r, k)$.\\
\item The positive integer exponent $n$ cannot exist always, so that \(r^{n} \leq W(r, k) < r^{n + 1}\) is true for each and every van der Waerden number $W(r, k)$ and when this number is expanded into some powers of $r$ as
\begin{equation}
W(r, k) = b_{n}r^{n} + b_{n - 1}r^{n - 1} + \cdots + b_{0}.
\end{equation}
\\
\item \(W(r, k) < r^{n + 1} \leq r^{k^{2}}\) cannot be true always for each and every van der Waerden number $W(r, k)$ if \(k \geq \sqrt{n + 1}\), where for all such positive integers $k$, the interval $[1, W(r, k)]$ has an AP of $k$ terms. 
\end{enumerate}
But which of these statements then, if any, would a reader conclude is true?
\indent \emph{Our approach does not mean that all mathematical proofs should not be deep and ought to be simple}. Such a rigid paradigm would be impossible to uphold. Yet we do view these results as being an example as to how the application of known properties about all integers--such as the fact that any integer $N$ greater than $r$ can be expanded into powers of $r$, where $n$ is the least integer exponent for which $r^{n}$ will divide $W(r, k)$, can lead to concrete results, such as the finding of the finer upper bound $r^{k^{2}}$ on $W(r, k)$.

\pagebreak

\end{document}